\documentclass[12pt]{article}
\usepackage{amssymb}
\usepackage{graphicx}
\usepackage[all]{xy}
\usepackage{color}
\usepackage{latexsym}
\usepackage{amssymb,amsmath,amstext}
\usepackage{epsfig}
\usepackage{graphics}
\usepackage{euscript}
\usepackage{ifthen}
\usepackage{wrapfig}
\usepackage{amsthm}
\usepackage{multicol}
\usepackage{paralist}
\usepackage{verbatim}
\usepackage[colorlinks=true, urlcolor=blue, linkcolor=blue, citecolor=blue]{hyperref}
\usepackage[margin=1in]{geometry}
\usepackage{algorithmicx}
\usepackage{algpseudocode}
\usepackage{algorithm}
\usepackage{placeins}
\usepackage{caption}
\usepackage{subcaption}
\usepackage{longtable}
\usepackage{lscape}
\usepackage{authblk}

\usepackage[makeroom]{cancel}

\numberwithin{equation}{section}
\theoremstyle{plain}%
\newtheorem{theorem}{Theorem}
\numberwithin{theorem}{section}

\newtheorem{example}[theorem]{Example}
\newtheorem{lemma}[theorem]{Lemma}

\newtheorem{definition}[theorem]{Definition}
\newtheorem{remark}[theorem]{Remark}
\newtheorem{conjecture}[theorem]{Conjecture}

\allowdisplaybreaks

\begin{document}

\title{\bf Orthogonal Decomposition of Symmetric Tensors}
\author{Elina Robeva}
\affil{University of California, Berkeley}
\date{}
\maketitle

\begin{abstract}
A real symmetric tensor is orthogonally decomposable (or odeco) if it can be written as a linear combination of symmetric powers of $n$ vectors which form an orthonormal basis of $\mathbb R^n$. Motivated by the spectral theorem for real symmetric matrices, we study the properties of odeco tensors. We give a formula for all of the eigenvectors of an odeco tensor. Moreover, we formulate a set of polynomial equations that vanish on the odeco variety and we conjecture that these polynomials generate its prime ideal. We prove this conjecture in some cases and give strong evidence for its overall correctness.
\end{abstract}

\section{Introduction}
The spectral theorem states that every $n\times n$ real symmetric matrix $M$ possesses $n$ real eigenvectors $v_1,\dots ,v_n$ which form an orthonormal basis of $\mathbb R^n$. Moreover, one can express $M$ as $M = \sum_{i=1}^n \lambda_1 v_i v_i^T$, where $\lambda_1,\dots ,\lambda_n\in\mathbb R$ are the corresponding eigenvalues.
In this paper we investigate when such a decomposition is possible for real symmetric tensors. We address the following two questions.

\bigskip

\noindent \textbf{Question 1.}\textit{ Which real symmetric tensors $T$ can be decomposed as $T = \lambda_1 v_1^{\otimes d} + \cdots + \lambda_n v_n^{\otimes d}$, form some orthonormal basis $ v_1,\dots ,v_n$ of $\mathbb R^n$ and some $\lambda_1,\dots,\lambda_n\in\mathbb R$? More precisely, can we find equations in the entries of $T$ that cut out the set of tensors for which such a decomposition exists?}

\bigskip
\noindent \textbf{Question 2.} \textit{Given that a tensor $T$ can be decomposed as  $T= \lambda_1v_1^{\otimes d} + \cdots + \lambda_n v_n^{\otimes d}$, where $v_1,\dots ,v_n\in\mathbb R^n$ are orthonormal, can we express the eigenvectors of $T$ (to be defined) in terms of $v_1,\dots ,v_n$?}
\bigskip

Let $S^d\left(\mathbb R^n\right)$ denote the space of $n\times n\times \cdots \times n$ ($d$ times) symmetric tensors, i.e. tensors whose entries are real numbers $T_{i_1\dots i_d}$ invariant under permuting the indices: $T_{i_1\dots i_d} = T_{i_{\sigma\left(1\right)}\dots i_{\sigma\left(d\right)}}$ for all permutations $\sigma$ of the set $\{1,2,\dots ,d\}$. For example, when $d=2$, the space $S^2\left(\mathbb R^n\right)$ consists of all $n\times n$ real symmetric matrices.
 We study the elements $T\in S^d\left(\mathbb R^n\right)$ which can be written as $T = \lambda_1 v_1^{\otimes d} + \cdots + \lambda_n v_n^{\otimes d}$, where $v_1,\dots ,v_n\in\mathbb R^n$ form an orthonormal basis of $\mathbb R^n$ and $\lambda_1,\dots,\lambda_n\in\mathbb R$. We call such tensors $T$ {\em orthogonally decomposable} or, for short, {\em odeco}. 
 
The notion of eigenvectors of matrices was extended to symmetric tensors by Lim~\cite{Lim} and by Qi \cite{Qi} independently in 2005. A vector $w\in\mathbb C^n$ is an {\em eigenvector} of $T\in S^d\left(\mathbb C^n\right)$ if there exists $\lambda\in\mathbb C$, the corresponding {\em eigenvalue}, such that
$$Tw^{d-1} := \Big[\sum_{i_2, \dots , i_d = 1}^nT_{i, i_2, \dots , i_d}w_{i_2}\dots w_{i_d}\Big]_i = \lambda w.$$
Two {\em eigenpairs} $\left(w, \lambda\right)$ and $\left(w', \lambda'\right)$ are equivalent if there exists $t\neq 0$ such that $w = t w'$ and $\lambda = t^{d-2}\lambda'$. 
When $d=2$, these definitions agree with the usual definitions of eigenvectors, eigenvalues, and equivalence of eigenpairs for matrices.

The spectral theorem answers both Question 1 and Question 2 in the case $d=2$: every symmetric matrix $M\in S^2\left(\mathbb R^n\right)$ can be written as $M = \sum_{i=1}^n \lambda_i v_i v_i^T =\sum_{i=1}^n \lambda_i v_i^{\otimes 2}$, where $v_1,\dots ,v_n$ are orthonormal. Moreover, if $M$ is generic (in the sense that its eigenvalues are distinct), then $v_1,\dots ,v_n$ are {\em all} of the eigenvectors of $M$ up to scaling.
\bigskip

In Section \ref{sec:eigenvectors} we give an explicit algebraic formula of all of the eigenvectors of an odeco tensor $T = \lambda_1 v_1^{\otimes d} + \cdots + \lambda_n v_n^{\otimes d}$ in terms of $v_1,\dots ,v_n$, answering Question 2 above. It easily follows from the definition of eigenvectors that $v_1,\dots ,v_n$ are eigenvectors of $T$. These are not all of the eigenvectors of $T$, but it turns out that one can explicitly express the rest of them in terms of $v_1, \dots , v_n$.

For general $d$, not all tensors $T\in S^d\left(\mathbb R^n\right)$ are odeco. In Section \ref{sec:odeco}, we address Question~1. We study the set of all odeco tensors and find equations that vanish on this set.  In Conjecture~\ref{conj:equations} we claim that these define the prime ideal of the odeco variety, which is the Zariski closure of the set of odeco tensors inside $S^d(\mathbb C^n)$. In Theorem \ref{prop:n=2} we prove Conjecture \ref{conj:equations} for the special case $n=2$. In Section \ref{sec:evidence} we conclude the paper by giving evidence for the correctness of this conjecture.

\bigskip
In the remainder of this section we review symmetric tensor decomposition as well as the equivalent characterization of symmetric tensors as homogeneous polynomials. We conclude the section by describing an algorithm, called the tensor power method, which finds the orthogonal decomposition of an odeco tensor. 

\subsection{Symmetric tensor decomposition}

Orthogonal decomposition is a special type of {\em symmetric tensor decomposition} which has been of much interest in the recent years; references include \cite{BCMT,LanO,LO,R}, and many others. Given a tensor $T\in S^d\left(\mathbb C^n\right)$, the aim is to decompose it as
$$T = \sum_{i=1}^r\lambda_i v_i^{\otimes d},$$
where $v_1,\dots ,v_r\in\mathbb C^n$ are any vectors and $\lambda_1,\dots,\lambda_r\in\mathbb C$. The smallest $r$ for which such a decomposition exists is called the {\em (symmetric) rank} of $T$. Finding the symmetric decomposition of a given tensor $T$ is an NP hard problem \cite{HL} and algorithms for it have been proposed by several authors, for example \cite{BCMT, LO}.

The rank of a generic tensor $T$ is $\binom{n+d-1}{d}$. However, the rank of an odeco tensor $T\in S^d(\mathbb R^n)$ is at most $n$. This means that the set of odeco tensors is a small subset of the set of all tensors. We explore this further in Section~\ref{sec:odeco}.

\begin{remark} Orthogonal tensor decomposition has also been studied in the non-symmetric case \cite{K03,K01}. An odeco tensor is also orthogonally decomposable according to the definition in the non-symmetric case.
\end{remark}

\subsection{Symmetric tensors as homogeneous polynomials}
An equivalent way to think about a symmetric matrix $M\in S^2\left(\mathbb C^n\right)$ is via its corresponding quadratic form $f_M\in \mathbb C[x_1,\dots ,x_n]$ given by
$$f_M\left(x_1,\dots ,x_n\right) = x^TMx = \sum_{i,j} M_{ij}x_ix_j.$$
More generally, a tensor $T\in S^d\left(\mathbb C^n\right) $ can equivalently be represented by a homogeneous polynomial $f_T\in\mathbb C[x_1, \dots , x_n]$ of degree $d$ given by
\begin{align*}
f_T\left(x_1, \dots , x_n\right) = T\cdot x^d :=& \sum_{i_1, \dots , i_d = 1}^n T_{i_1, \dots , i_d}x_{i_1}x_{i_2}\dots x_{i_d}.
\end{align*}
Given $T\in S^d\left(\mathbb C^n\right)$, we can describe the notions of eigenvectors, eigenvalues, and symmetric decomposition in terms of the corresponding polynomial $f_T\in\mathbb C[x_1,\dots ,x_n]$ as follows.

A vector $x\in\mathbb C^n$ is an eigenvector of $T$ with eigenvalue $\lambda$ if and only if
$$\nabla f_T\left(x\right) = \lambda d x.$$
The tensor $T$ can be decomposed as $T = \sum_{i=1}^r\lambda_i v_i^{\otimes d}$ if and only if the corresponding polynomial $f_T$ can be decomposed as
$$f_T\left(x_1,\dots ,x_n\right) = \sum_{i=1}^r \lambda_i \left(v_{i1}x_1 + \cdots + v_{in}x_n\right)^d.$$
Similarly, a real tensor $T\in S^d\left(\mathbb R^n\right)$ is orthogonally decomposable with $T = \lambda_1v_1^{\otimes d} + \cdots + \lambda_n v_n^{\otimes d}$, where $\lambda_1,\dots ,\lambda_k\in\mathbb R$ and $v_1,\dots ,v_k\in\mathbb R^n$ are orthonormal, if and only if $f_T\left(x_1,\dots ,x_n\right) = \lambda_1 \left(v_1\cdot x\right)^d + \cdots + \lambda_n\left(v_n \cdot x\right)^d$.

This equivalent characterization of symmetric tensors as homogeneous polynomials proves to be quite useful in the sequel.

\subsection{Finding an orthogonal decomposition}\label{sec:motivation}

Finding the symmetric decomposition of a given $T\in S^d\left(\mathbb C^n\right)$ is NP hard~\cite{HL}. However, there are simple algorithms that recover the orthogonal decomposition of an odeco tensor $T\in S^d\left(\mathbb R^n\right)$. One such algorithm is the {\em tensor power method} \cite{AGHKT}.

 Let $T\in S^d\left(\mathbb R^n\right)$. If $T$ is orthogonally decomposable, i.e. $T = \sum_{i=1}^k \lambda_i v_i^{\otimes d}$ and $v_1,\dots ,v_k\in\mathbb R^n$ orthonormal, then 
$$T\cdot v_j^{d-1} = \sum_{i=1}^r \lambda_i \left(v_i\cdot v_j\right)^{d-1}v_i = \lambda_jv_j,$$
 for all $j=1,2,\dots ,k$. Thus, $v_1,\dots ,v_k$ are eigenvectors of $T$ with corresponding eigenvalues $\lambda_1,\dots ,\lambda_k$. Note that requiring $T$ and $v_1,\dots,v_k$ to be real forces $\lambda_1,\dots,\lambda_k$ to be real as well.
\begin{definition}
A unit vector $u\in\mathbb R^n$ is a {\em robust eigenvector} of $T\in S^d\left(\mathbb R^n\right)$ if there exists $\epsilon > 0$ such that for all $\theta\in\{u'\in\mathbb R^n : \|u - u'\| < \epsilon\}$, repeated iteration of the map
\begin{align}
\overline{\theta} \mapsto \frac{T\overline{\theta}^{d-1}}{\|T\overline{\theta}^{d-1}\|}, \label{map}
\end{align}
starting from $\theta$ converges to $u$.
\end{definition}
The following theorem shows that if $T$ has an orthogonal decomposition $T = \sum_{i=1}^k\lambda_i v_i^{\otimes d}$, then the set of robust eigenvectors of $T$ is precisely the set $\{v_1, v_2, \dots , v_k\}$, implying that the orthogonal decomposition is unique up to the obvious reordering.
\begin{theorem}[Theorem~4.1, \cite{AGHKT}]\label{thm:robust} Let $T\in S^d\left(\mathbb R^n\right)$ have an orthogonal decomposition $T = \sum_{i=1}^k \lambda_i v_i^{\otimes d}$, where $v_1,\dots ,v_k\in\mathbb R^n$ are orthonormal.
\begin{enumerate}
\item The set of $\theta\in\mathbb R^n$ which do not converge to some $v_i$ under repeated iteration of $\left(\ref{map}\right)$ has measure $0$.
\item The set of robust eigenvectors of $T$ is equal to $\{v_1, v_2, \dots , v_k\}$.
\end{enumerate}
\end{theorem}
Therefore, to recover the orthogonal decomposition of $T$, one needs to find the robust eigenvectors. The definition of robust eigenvectors suggests an algorithm to compute them, using repeated iteration of the map (\ref{map}) starting with random vectors $u\in\mathbb R^n$. 
\begin{algorithm}
\caption{The Tensor Power Method}
\begin{algorithmic}[1]
\State \textbf{Input:} an orthogonally decomposable tensor $T$.
\State Set $i=1$.
\State \textbf{Repeat} until $T = 0$.
\State\indent    Choose random $u\in\mathbb R^m$.
\State\indent    Let $v_i$ be the result of repeated iteration of (\ref{map}) starting with $u$.
\State\indent    Compute the eigenvalue $\lambda_i$ corresponding to $v_i$, from the equation $Tv_i^{d-1} = \lambda_i v_i$.
\State\indent    Set $T = T - \lambda_i v_i^{\otimes d}$.
\State\indent $i \leftarrow i+1$.
\State \textbf{Output} $v_1,\dots ,v_k$ and $\lambda_1,\dots ,\lambda_k$.
\end{algorithmic}
\end{algorithm}
%

In certain cases, this algorithm can be used to find the symmetric decomposition of a given tensor.  For example, the authors of \cite{AGHKT} consider a class of statistical models, such as the exchangeable single topic model, in which one observes tensors $T_2$ and $T_3$, where $T_d = \sum_{i=1}^k \omega_i \mu_i^{\otimes d}$ for $d = 2,3$ and the aim is to recover the unknown parameters $\omega = \left(\omega_1,\dots ,\omega_k\right)\in\mathbb R^k$ and $\mu_1,\dots ,\mu_k\in \mathbb R^n$.  (Note that $T_2$ and $T_3$ have decompositions using the same vectors and observing both of them gives more information than observing only $T_3$). This is done by transforming $T_2$ and $T_3$ (in an invertible way) into orthogonally decomposable tensors $\tilde{T}_2$ and $\tilde T_3$, where $\tilde{T}_d = \sum_{i=1}^k\tilde{\omega}_i \tilde{\mu}_i^{\otimes d}$ and $\tilde{\mu}_1,\dots ,\tilde{\mu}_k$ are orthonormal, $d=2,3$. Then, they use the tensor power method to find $\tilde{\mu}_1,\dots ,\tilde{\mu}_k$ and $\tilde{\omega}_1,\dots ,\tilde{\omega}_k$ and use the inverse transformation to recover the original $\mu_1,\dots ,\mu_k$ and $\omega_1,\dots ,\omega_k$.

\begin{remark} As mentioned above, Theorem \ref{thm:robust} also implies that an odeco tensor $T$ has a unique orthogonal decomposition. That is because the elements in the orthogonal decomposition are uniquely determined as the robust eigenvectors $v_1,\dots, v_k$ and the corresponding constants $\lambda_1,\dots, \lambda_k$ are uniquely determined by $\lambda_i = T\cdot v_i^d$.
\end{remark}

Another method, described in \cite{BCMT}, can also be used to efficiently compute the decomposition of a symmetric tensor $T$ of rank at most $n$. It involves computing generalized eigenvectors of sub-matrices of the Hankel matrices associated to $T$.

\section{The Variety of Eigenvectors of a Tensor}\label{sec:eigenvectors}

In this section, we are going to study the set of all eigenvectors of a given orthogonally decomposable tensor.

As we mentioned in the introduction, a symmetric tensor $T\in S^d\left(\mathbb R^n\right)$ can equivalently be represented by a homogeneous polynomial $f_T\in\mathbb R[x_1, \dots , x_n]$ of degree $d$. Indeed, given $T$, we obtain $f_T$ by
$$f_T\left(x_1,\dots ,x_n\right) = \sum_{i_1, \dots , i_d}T_{i_1, \dots , i_d}x_{i_1}\cdots x_{i_d}.$$
Then, for $x\in\mathbb C^n$, $Tx^{d-1} = \lambda x$ is equivalent to $\nabla f_T\left(x\right) = d \lambda x$, i.e. $\nabla f_T\left(x\right)$ and $x$ are parallel to each other. This is equivalent to the vanishing of the $2\times 2$ minors of the $n\times 2$ matrix $\begin{bmatrix}\nabla f_T\left(x\right)  \big|  x\end{bmatrix}$.
\begin{definition} The {\em variety of eigenvectors} $\mathcal V_T$ of a given symmetric tensor $T$ with corresponding polynomial $f_T$ is the zero set of the $2\times 2$ minors of the matrix $\begin{bmatrix}\nabla f_T\left(x\right)  \big|  x\end{bmatrix}$.
\end{definition}

\begin{remark}
Consider the gradient map as a map on projective spaces:
$$\nabla f_T: \mathbb {CP}^{n-1} \to \mathbb {CP}^{n-1}$$
$$[x]\mapsto [\nabla f_T\left(x\right)].$$
Then, the eigenvectors of $f_T$ are precisely the fixed points of $\nabla f_T$. This map is well-defined provided the hypersurface $\{f_T = 0\}$ has no singular points.
\end{remark}

\noindent The aim of this section is to prove the following theorem.
\begin{theorem}\label{eigenvectorRepresentation}
Let $T\in S^d\left(\mathbb R^n\right)$ be odeco with $f_T\left(x\right) = \sum_{i=1}^l \lambda_i \left(v_i\cdot x\right)^d$, where $v_1,\dots ,v_l\in\mathbb R^n$ are orthonormal. Assume that $1\leq l\leq n$ and $\lambda_1,\dots ,\lambda_l\neq 0$ . Then, $T$ has $\frac{\left(d-1\right)^l-1}{d-2}$ eigenvectors in $\mathbb C^n$, given explicitly in terms of $v_1,\dots ,v_l$ and the $\left(d-2\right)$-nd roots of $\lambda_1,\dots ,\lambda_l$ as follows. Let  $V = \begin{bmatrix}-&v_1& - \\
& \vdots & \\
- & v_l & -\end{bmatrix}\in\mathbb R^{l\times n}$. Then, for any $1\leq k\leq l$, any $\mathcal I = \{i_1, i_2,\dots ,i_k\}\subseteq [l]$ and any $\left(k-1\right)$-tuple $\eta_{1},\dots ,\eta_{k-1}$ of $\left(d-2\right)$-nd roots of unity, there is one eigenvector $w$, up to scaling, where $w = V^T\left(y_1,\dots ,y_l\right)^T$ and
$$y_i = \begin{cases} \eta_{j}\lambda_{i_j}^{-\frac1{d-2}} & \text{ if } i = i_j \text{ and }j\in\{1,\dots ,k-1\}\\
\lambda_{i_k}^{-\frac1{d-2}} & \text{ if } i = i_k\\
0 & \text{ if }i\not\in \mathcal I.
\end{cases}$$
The rest of the eigenvectors are all the elements in the nullspace of $V$.

\end{theorem}

\begin{remark} It is known by \cite{CS} that if a tensor $T\in S^d\left(\mathbb R^n\right)$ has finitely many equivalence classes of eigenpairs $\left(x, \lambda\right)$ over $\mathbb C$, then their number, counted with multiplicity, is equal to $\frac{\left(d-1\right)^n-1}{d-2}$. If the entries of $T$ are sufficiently generic, then all multiplicities are equal to 1, so there are exactly $\frac{\left(d-1\right)^n-1}{d-2}$ equivalence classes of eigenpairs.

In the proof of Theorem \ref{eigenvectorRepresentation} we independently show that an odeco tensor $T$ with orthogonal decomposition $T = \lambda_1v_1^{\otimes d} + \cdots + \lambda_n v_n^{\otimes d}$, such that $\lambda_1,\dots ,\lambda_n\neq 0$ has  finitely many equivalence classes of eigenvectors and their number is exactly $\frac{\left(d-1\right)^n-1}{d-2}$. 

\end{remark}

%

We illustrate Theorem \ref{eigenvectorRepresentation} by two simple concrete examples.
\begin{example}
Let $d = n = 3$ and consider the odeco tensor $T$ with polynomial form
$$f_T\left(x,y,z\right) = \lambda_1x^3 + \lambda_2y^3 + \lambda_3z^3.$$
This type of polynomial is called a Fermat polynomial. In this case $v_1 = \left(1,0,0\right), v_2 = \left(0,1,0\right), v_3 = \left(0,0,1\right)$ and the matrix $V = I$. Since $d-2 =1$, taking the $\left(d-2\right)$-nd root is the identity map. Thus, the eigenvectors of $T$ are as follows.

When $k = 1$, $\mathcal I = \{1\}, \{2\},$ or $\{3\}$. The corresponding three eigenvectors are
$$\left(\frac 1{\lambda_1}, 0, 0\right)^T, \left(0, \frac 1{\lambda_2}, 0\right)^T, \left(0,0, \frac 1{\lambda_3}\right)^T.$$
When $k=2$, $\mathcal I = \{1,2\}, \{1,3\},$ or $\{ 2,3\}$. The corresponding eigenvectors are
$$\left(\frac 1{\lambda_1}, \frac 1{\lambda_2}, 0\right)^T, \left(\frac 1{\lambda_1}, 0, \frac 1{\lambda_3}\right)^T, \left(0, \frac 1{\lambda_2}, \frac 1{\lambda_3}\right)^T.$$
When $k=3$, $\mathcal I = \{1,2,3\}$ and the corresponding eigenvector is
$$\left(\frac 1{\lambda_1}, \frac 1{\lambda_2}, \frac 1{\lambda_3}\right)^T.$$

Figure 1 shows what these eigenvectors look like geometrically.
\end{example}

\begin{figure}
\begin{center}
\includegraphics[width=0.6\textwidth]{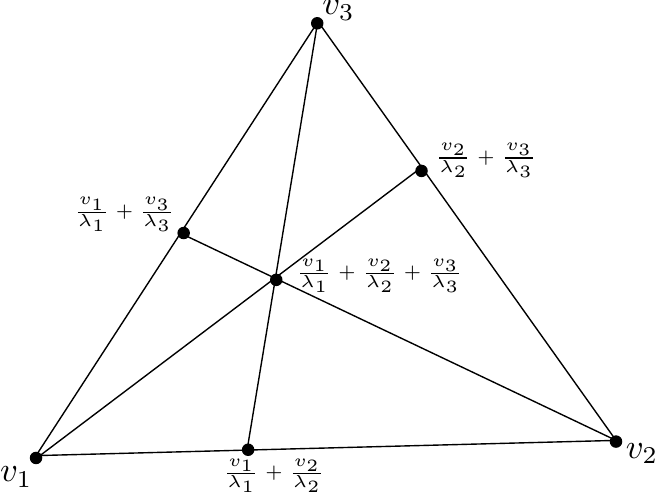}
\caption{This figure shows the structure of the eigenvectors inside $\mathbb {CP}^2$ of an odeco tensor $T\in S^3\left(\mathbb R^3\right)$ such that $T = \lambda_1v_1^{\otimes 3} + \lambda_2 v_2^{\otimes 3} + \lambda_3v_3^{\otimes 3}$ with $\lambda_1,\lambda_2,\lambda_3\neq 0$.}

\end{center}
\end{figure}

\begin{example} Let $d = 4, n=4$ and consider $T\in S^4(\mathbb R^4)$ with corresponding polynomial
$$f_T(x_1,\dots,x_4) = x_1^4 + 2x_2^4.$$
In the notation of Theorem \ref{eigenvectorRepresentation}, the number of nonzero coefficients is $l=2 < n$. We have that $v_1 = e_1, v_2=e_2$ and $\lambda_1=1, \lambda_2=2$. Since $d-2 = 2$, the roots $\eta_i$ can be $\pm 1$. Thus, the eigenvectors of $T$ are as follows.

When $k=1$, $\mathcal I = \{1\},\{2\}$. The corresponding eigenvectors are
$$(1,0,0,0)^T,(0,\frac1{\sqrt 2},0,0)^T.$$

When $k=2$, $\mathcal I = \{1,2\}$. The corresponding eigenvectors are
$$(1,\frac1{\sqrt 2},0,0)^T, (-1,\frac1{\sqrt 2}, 0, 0)^T.$$

The rest of the eigenvectors are all vectors perpendicular to $e_1$ and $e_2$, i.e.
$$(0,0,a,b)^T$$
for any $a,b\in\mathbb C$ not both zero.
\end{example}

%
%
%

In the rest of this section we prove Theorem \ref{eigenvectorRepresentation}. We proceed as follows. First we show that the theorem is valid when $f_T = \lambda_1 x_1^{d} + \cdots + \lambda_n v_n^{d}$, where $\lambda_1,\dots ,\lambda_n\neq 0$. This is done in Lemma  \ref{lemma:FermatHypersurface}. For the general case, $f_T = \lambda_1\left(v_1\cdot x\right)^d + \cdots + \lambda_l \left(v_l\cdot x\right)^d$, where $\lambda_1,\dots ,\lambda_l\neq 0$ and $v_1,\dots ,v_l$ are orthonormal, we observe that setting $y_i = v_i\cdot x$ the eigenvectors of the Fermat polynomial tensor $\lambda_1 y_1^d + \cdots + \lambda_l y_l^d$ are in a 1-to-1 correspondence with some of the eigenvectors of $T$ via the transformation given by the matrix $V$ with rows $v_1,\dots ,v_l$. This is how we recover the formula in Theorem \ref{eigenvectorRepresentation}.

\begin{definition} Given $f\left(x_1,\dots ,x_n\right) = \lambda_1x_1^d + \cdots + \lambda_n x_n^d$, $\mathcal I = \{i_1,\dots ,i_k\}\subseteq\{1,2,\dots ,n\}$, and $\eta = \{\eta_1,\dots ,\eta_{k-1}\}$ such that $\eta_1,\dots ,\eta_{k-1}$ are $\left(d-2\right)$-nd roots of unity, and define the ideal
$$I_{\mathcal I, \eta} = \langle \lambda_{i_1}^{\frac 1{d-2}} x_{i_1} - \eta_1\lambda_{i_k}^{\frac1{d-2}}, \dots , \lambda_{i_{k-1}}^{\frac1{d-2}}x_{i_{k-1}} - \eta_{k-1}\lambda_{i_k}^{\frac1{d-2}} x_{i_k} \rangle + \langle x_j | j\not\in \mathcal I\rangle$$
in the polynomial ring $\mathbb C[x_1,\dots ,x_n]$.
\end{definition}
\begin{lemma}\label{lemma:FermatHypersurface}
Theorem \ref{eigenvectorRepresentation} is true in the case $f_T\left(x_1, \dots , x_n\right) = \lambda_1 x_1^d + \lambda_2 x_2^d + \cdots + \lambda_n x_n^d$, where $\lambda_1,\dots ,\lambda_n\neq 0$. In particular,
the radical of the ideal $I$ of $2\times 2$ minors of $\begin{bmatrix}\nabla f(x) | x\end{bmatrix}$ can be decomposed as follows.
\begin{align}\label{idealDecomposition}\sqrt{I} =\bigcap_{\mathcal I\subseteq[n], \eta=\{\eta_1,\dots ,\eta_{|\mathcal I|-1}\}}I_{\mathcal I, \eta},\end{align}
where $\eta_1, \dots , \eta_{k-1}$ are $\left(d-2\right)$-nd roots of unity. For every $k\in\{1,\dots, n\}$, there are $\binom n k \left(d-2\right)^{k-1}$ homogeneous prime ideals $I_{\mathcal I, \eta}$ with $|\mathcal I| = k$. Each ideal $I_{\mathcal I, \eta}$ has exactly one solution in $\mathbb C\mathbb P^{n-1}$, representing one eigenvector, namely $w = \left(w_1: \dots :w_n\right)$ such that
$$w_i = \begin{cases} \eta_{l} \frac1{\lambda_{i_l}}^{-\frac1{d-2}} & \text{ if } i = i_l \text{ and } l\leq k-1,\\\
\lambda_{i_k}^{-\frac1{d-2}}& \text{ if } i = i_k, \\
0 & \text{ if } i\not\in \mathcal I.
\end{cases}$$
The total number of such solutions is $\frac{\left(d-1\right)^n - 1}{d-2}$.
\end{lemma}
\begin{proof}
Note that in this case, up to a factor of $d$ in the first row, we have that
$$\begin{bmatrix}\nabla f\left(x\right)  \big|  x\end{bmatrix} = \begin{bmatrix}
\lambda_1x_1^{d-1} & x_1\\
\lambda_2 x_2^{d-1} & x_2\\
\vdots & \vdots\\
\lambda_n x_n^{d-1} & x_{n}
\end{bmatrix}$$
Therefore, the ideal of $2\times 2$ minors is given by
$$I = \langle x_i x_j\left(\lambda_i x_i^{d-2} - \lambda_j x_j^{d-2}\right) : i\neq j \rangle.$$
We would like to decompose the variety of this ideal. Note that for any primary ideal $P\supseteq I$ its associated prime $\sqrt P$ would either contain $x_ix_j$ or $\lambda_i x_i^{d-2} - \lambda_j x_j^{d-2}$ for all $i\neq j$. Suppose that for a given $P\supseteq I$, $\sqrt P$ contains exactly $n-k$ of the variables $x_1, \dots , x_n$. Let $\mathcal I = \{i_1,\dots ,i_k\}\subseteq [n]$ and assume that $\sqrt P$ contains exactly those $x_i$ for which $i\not\in \mathcal I$. Thus, $\sqrt P$ also contains $\lambda_i x_i^{d-2} - \lambda_j x_j^{d-2}$ for $i\neq j, i,j \in\mathcal I$. 
Moreover, we can write $\sqrt P$ as $\sqrt P = \langle x_i : i\not\in \mathcal I\rangle + \sqrt P\cap \mathbb C[x_i: i\in \mathcal I]$. Then, the ideal $\sqrt P\cap \mathbb C[x_i: i\in \mathcal I]$ is prime, it doesn't contain $x_i$ for $i\in\mathcal I$ and contains $I_{\mathcal I}\subseteq \mathbb C[x_i: i\in\mathcal I]$, where
$$I_{\mathcal I} := \langle \lambda_i x_i^{d-2} - \lambda_j x_j^{d-2} : i\neq j, i,j \in\mathcal I\rangle = \langle \lambda_{i_j} x_{i_j}^{d-2} - \lambda_{i_{j+1}}x_{i_{j+1}}^{d-2} : j=1,\dots ,k-1\rangle.$$
Therefore, $\sqrt P\cap  \mathbb C[x_i: i\in \mathcal I]$ is a prime ideal containing $(I_{\mathcal I} : \langle x_i: i\in\mathcal I\rangle^{\infty})$.

We now describe the decomposition of the ideal $(I_{\mathcal I} : \langle x_i: i\in\mathcal I\rangle^{\infty})$ following Theorem 2.1 and Corollary 2.5 in \cite{ES}. Recall that $\mathcal I = \{i_1,\dots,i_k\}\subseteq [n]$. Let $L_\rho := \langle \left(d-2\right)\left(e_{i_j} - e_{i_k}\right) : j=1,\dots ,k-1\rangle$ be a lattice with {\em partial character} $\rho:L_{\rho} \to \mathbb C^*$ given by
$$\rho\left(\left(d-2\right)\left(e_{i_j} - e_{i_k}\right)\right) = \frac{\lambda_{i_k}}{\lambda_{i_j}}.$$
For any partial character $\sigma : L_{\sigma} \to \mathbb C^*$, define the ideal $I_+(\sigma) := \langle x^{m_+} - \sigma\left(m\right)x^{m_-}: m\in L_{\sigma}\rangle$, where $m = m_+ - m_-$ and $m_+, m_-$ have nonnegative entries. From this definition, we see that 
$$I_+(\rho) = (I_{\mathcal I} : \langle x_i: i\in\mathcal I\rangle^{\infty}).$$

Then, by Corollary 2.5 in \cite{ES}, the decomposition of $(I_{\mathcal I}:\langle x_i:i\in\mathcal I\rangle^{\infty})=I_+\left(\rho\right)$ is 
$$(I_{\mathcal I}:\langle x_i:i\in\mathcal I\rangle^{\infty} )= \bigcap_{\rho' \text{ extends } \rho \text{ to } L } I_+\left(\rho'\right),$$
where $L$ is a sublattice of $\mathbb Z^n$ such that $L_{\rho}\subseteq L \subseteq \mathbb Z^n$ and $|L/L_{\rho}|$ is finite.
In this case, we can choose
$$L = \langle e_{i_j} - e_{i_k} : j = 1,\dots ,k-1\rangle.$$
Then, $|L/L_{\rho}| = \left(d-2\right)^{k-1}$. Moreover, by the same theorem, the number of $\rho'$ extending $\rho$ is exactly $|L/L_{\rho}| = \left(d-2\right)^{k-1}$. Also, note that each such $\rho' : L\to \mathbb C^*$ is uniquely defined by the values
$$\eta_j \left(\frac{\lambda_{i_k}}{\lambda_{i_j}}\right)^{\frac1{d-2}} := \rho'\left(e_{i_j} - e_{i_k}\right)$$
for some $\left(d-2\right)$-nd root of unity $\eta_j$. Therefore,
$$I_+\left(\rho'\right) = \left\langle x_{i_j} - \eta_j \left(\frac{\lambda_{i_k}}{\lambda_{i_j}}\right)^{\frac1{d-2}} x_{i_k} : j = 1,2,\dots ,k-1\right\rangle$$
and each such ideal is maximal inside $\mathbb C[x)i:i\in\mathcal I]$.
Thus, the prime $\sqrt P\cap \mathbb C[x_i:i\in \mathcal I]$ must contain one of the ideals $I_+(\rho')$.
Therefore, $\sqrt P$ contains $\langle x_i:i\not \in\mathcal I\rangle + I_+(\rho')$ for some $\rho'$. But this ideal is maximal in $\mathbb C[x_1,\dots, x_n]$, therefore, $\sqrt P = \langle x_i:i\not \in\mathcal I\rangle + I_+(\rho')$.

Therefore, \eqref{idealDecomposition} holds and the minimal associated primes of the ideal $I$ are 
$$I_{\mathcal I, \eta} = \langle x_i:i\not \in\mathcal I\rangle + \left\langle x_{i_j} - \eta_j \left(\frac{\lambda_{i_k}}{\lambda_{i_j}}\right)^{\frac1{d-2}} x_{i_k} : j = 1,2,\dots ,k-1\right\rangle,$$
where $\mathcal I=\{i_1,\dots,i_k\}\subseteq [n]$ and $\eta_1,\dots,\eta_{k-1}$ are $(d-2)$-nd roots of unity. Each ideal $I_{\mathcal I, \eta}$ is zero-dimensional and corresponds to one eigenvector $w = (w_1:\cdots : w_n)$, where
$$w_i = \begin{cases}
  \eta_{l} \frac1{\lambda_{i_l}}^{-\frac1{d-2}} & \text{ if } i = i_l \text{ and } l\leq k-1,\\\
\lambda_{i_k}^{-\frac1{d-2}}& \text{ if } i = i_k, \\
0 & \text{ if } i\not\in \mathcal I.
 \end{cases}$$
%
%
Moreover, since there are $\binom nk$ options for choosing $\mathcal I\subseteq [n]$ with $|\mathcal I | = k$ and $\left(d-2\right)^{k-1}$ options for choosing $\eta = \left(\eta_1,\dots ,\eta_{k-1}\right)$, the total number of eigenvectors of $f$ is
\begin{align*}
&\sum_{k=1}^{n}\binom{n}{k} \left(d-2\right)^{k-1} = \frac1{d-2}\sum_{k=1}^n\binom{n}k \left(d-2\right)^k \\
&= \frac1{d-2}\left(\left(d-2+1\right)^n - 1\right) = \frac{\left(d-1\right)^n-1}{d-2},
\end{align*}
recovering the formula expected by \cite{CS}.
\end{proof}
\noindent Now, we proceed with the proof of Theorem \ref{eigenvectorRepresentation}.
\begin{proof}[Proof of Theorem \ref{eigenvectorRepresentation}]
Let $T = \sum_{i=1}^l \lambda_i v_i^{\otimes d}$ be odeco with $\lambda_1,\dots,\lambda_l\neq 0$. Then,
$$f_T\left(x\right) = \sum_{i=1}^l \lambda_i\left(v_i\cdot x\right)^d$$
and
$$\frac1 d\nabla f_T\left(x\right) = \sum_{i=1}^l \lambda_i\left(v_i\cdot x\right)^{d-1} v_i.$$
If $x\in\mathbb C^n$ is an eigenvector, then
$$\frac1 d\nabla f_T\left(x\right)= \sum_{i=1}^l \lambda_i\left(v_i\cdot x\right)^{d-1} v_i = \lambda x.$$
Let $v_{l+1}, \dots, v_n\in\mathbb R^n$ complete $v_1,\dots,v_l$ to an orthonormal basis of $\mathbb R^n$. Then, they are also a basis of $\mathbb C^n$ and $x = \sum_{i=1}^n \left(v_i\cdot x\right)v_i$ for any $x\in\mathbb C^n$, where $v_i\cdot x = \sum_j v_{ij}x_j$ is still the usual dot product on $\mathbb R^n$. Since the $v_i$ form a basis of $\mathbb C^n$ and
$$\sum_{i=1}^l \lambda_i\left(v_i\cdot x\right)^{d-1} v_i  = \lambda  \sum_{i=1}^n \left(v_i\cdot x\right)v_i,$$
then $x$ is an eigenvector if and only if the vectors $\left(\lambda_1\left(v_1\cdot x\right)^{d-1}, \dots , \lambda_l\left(v_n\cdot x\right)^{d-1},0,\dots,0\right)$ and $\left(v_1\cdot x, \dots , v_n\cdot x\right)$ are parallel. Let $\tilde{V} = \begin{bmatrix} -&v_1&-\\&\vdots&\\-&v_n&-\end{bmatrix}\in\mathbb R^{n\times n}$ be the orthogonal matrix whose rows are $v_1,\dots,v_n$.
Let
$$ y_i = \left(v_i \cdot x\right),{\text{ i.e. }} y = \tilde{V}x.$$
Then, an equivalent description of $x$ being an eigenvector is that $\left(\lambda_1 y_1^{d-1},\dots ,\lambda_l y_l^{d-1},0,\dots,0\right)$ and $y$ are parallel. In other words, the matrix 
$$\begin{bmatrix}\lambda_1 y_1^{d-1}&\cdots&\lambda_l y_l^{d-1} &0 &\cdots & 0 \\ y_1 &\cdots & y_l&y_{l+1}&\cdots&y_n \end{bmatrix}$$
has rank at most one. There are two cases.

\underline{Case 1:} One of the numbers $y_{l+1},\dots,y_n$ is nonzero. This forces $y_1=\cdots=y_l=0$ and any choice of $y_{l+1},\dots,y_n$ gives a solution. This means that any vector $x\in \text{span}\{v_1,\dots,v_l\}^\perp$ is an eigenvector of the original tensor $T$.

\underline{Case 2:} The other case is that $y_{l+1}=\cdots = y_n=0$. Then the above matrix having rank at most one is equivalent to the smaller matrix
$$\begin{bmatrix}\lambda_1 y_1^{d-1}&\cdots&\lambda_l y_l^{d-1}  \\ y_1 &\cdots & y_l \end{bmatrix}$$
having rank at most one. The ideal of the $2\times 2$ minors of this matrix is
$$I = \langle \lambda_i y_i^{d-1} y_j - \lambda_j y_j^{d-1} y_i : i< j\leq l\rangle.$$
By Lemma \ref{lemma:FermatHypersurface}, the radical of this ideal decomposes as
\begin{align*}
\sqrt I = \bigcap_{\mathcal I\subseteq [l], \eta} I_{\mathcal I, \eta}
\end{align*}
and each ideal $I_{\mathcal I, \eta}$ with $\mathcal  I = \{i_1,\dots ,i_k\}\subseteq [l]$ has the form
\begin{align}\label{yIdeals}
I_{\mathcal I, \eta}=\langle\lambda_{i_1}^{\frac1 {d-2}} y_{i_1} - \eta_1 \lambda_{i_k}^{\frac1{d-2}} y_{i_k}, \dots , \lambda_{i_{k-1}}^{\frac1{d-2}}y_{i_{k-1}} - \eta_{k-1}\lambda_{i_k}^{\frac1{d-2}}, y_{i_k}\rangle + \langle y_{i} : i\not\in\mathcal I\rangle,
\end{align}
where $\eta_1, \dots , \eta_{k-1}$ are $\left(d-2\right)$-nd roots of unity. By the Nullstellensatz, all elements in $\mathcal V(I)$ are the same as those in $\mathcal V(\sqrt I)$, which are in turn the elements in $\bigcup\mathcal V(I_{\mathcal I, \eta})$. Each ideal $I_{\mathcal I, \eta}$ gives exactly one solution in $\mathbb C\mathbb P^n$, representing one eigenvector $\left(y_1,\dots ,y_n\right)$ such that
\begin{align}
y_i = \begin{cases}
 \eta_{s} \frac1{\lambda_{i_s}}^{-\frac1{d-2}} & \text{ if } i = i_s \text{ and } s\leq k-1,\\\
\lambda_{i_k}^{-\frac1{d-2}}& \text{ if } i = i_k, \\
0 & \text{ if } i\in [n]\setminus\mathcal I.
\end{cases}
\label{yFormula}
\end{align}
Note that $y = \tilde Vx$ and $\tilde V$ is an orthogonal matrix. Therefore,
$$x =\tilde{V}^T y.$$
By Lemma \ref{lemma:FermatHypersurface}, we know that for each $k$ there are $\binom l k \left(d-2\right)^{k-1}$ eigenvectors with $k$ nonzero entries, which makes for a total of 
$$\sum_{k=1}^l\binom l k \left(d-2\right)^{k-1} =  \frac1{d-2}\left(\sum_{k=1}^l\binom n k \left(d-2\right)^{k}\right)$$
$$= \frac1{d-2}\left( \sum_{k=0}^l\binom n k \left(d-2\right)^{k}  -1 \right)=\frac{\left(d-1\right)^l - 1}{d-2}$$ eigenvectors of $T$ in this case.
\end{proof}
\section{The Odeco Variety}\label{sec:odeco}

The {\em odeco variety} is the Zariski closure in $S^d\left(\mathbb C^n\right)$ of the set of all tensors $T\in S^d\left(\mathbb R^n\right)$ which are orthogonally decomposable. If a tensor is odeco, then, in particular, its corresponding polynomial $f_T$ is decomposable as a sum of $n$ $d$-th powers of linear forms, i.e. it lies in the $n$-th secant variety of the $d$-th Veronese variety, denoted by $\sigma_n\left(v_d\left(\mathbb C^n\right)\right)$.

When $d = n = 3$, there is one equation defining $\sigma_3\left(v_3\left(\mathbb C^3\right)\right)$, called the Aronhold invariant \cite{L}, and it is given by the Pfaffian of a certain skew-symmetric matrix. The corresponding odeco variety in $S^3\left(\mathbb C^3\right)$ has codimension 4 and its prime ideal is generated by six quadrics, defined in Example~\ref{example:33}. For higher $d$ and $n$, the equations defining $\sigma_n\left(v_d\left(\mathbb C^n\right)\right)$ are much harder to compute. However, the odeco variety is smaller than $\sigma_n\left(v_d\left(\mathbb C^n\right)\right)$ and we believe that the defining equations of its prime ideal are quadrics that are easy to write down. They are shown in Conjecture \ref{conj:equations}.

\begin{lemma}\label{lem:dimension} The dimension of the odeco variety in $S^d\left(\mathbb C^n\right)$ is $\binom{n+1}2$.
\end{lemma}

\begin{proof}
Consider the map
$$\phi: \mathbb R^n\times SO_n \to S^d\left(\mathbb R^n\right) \subset S^d\left(\mathbb C^n\right)$$
given by
$$\left(\lambda_1,\dots ,\lambda_n\right), V \mapsto \sum_{i=1}^n\lambda_i v_i^{\otimes d},$$
where $v_i$ is the $i$th row of the orthogonal matrix $V$. The image Im$\left(\phi\right)$ of this map is precisely the set of orthogonally decomposable tensors in $S^d\left(\mathbb R^n\right)$. The odeco variety is $\overline{\text{Im}\left(\phi\right)}\subset S^d\left(\mathbb C^n\right)$.
Note that by Theorem \ref{thm:robust}, $\phi$ has a finite fiber (up to permutations of the input). Then, $\dim($Im$\left(\phi ) \right) = \dim\left(\mathbb R^n\times SO_n\right) = n + \binom n 2 = \binom{n+1}2$. Therefore, the dimension of the odeco variety is $\dim \left(\overline{\text{Im}\left(\phi\right)}\right) = \binom{n+1}2$.
\end{proof}

We are going to conjecture what the defining equations of the odeco variety are. In Theorem \ref{prop:n=2} we prove the result for the case $n=2$.

Consider a tensor $T\in S^d\left(\mathbb C^n\right)$ and the corresponding homogeneous polynomial $f_T(x_1, x_2,$ $ \dots , x_n)\in\mathbb C[x_1, \dots , x_n]$ of degree $d$. To define our equations, it is more convenient to work with the polynomial version of the tensor. As mentioned before, given $T\in S^d\left(\mathbb C^n\right)$, the corresponding polynomial can be rewritten as
$$f_T\left(x_1,\dots ,x_n\right) = \sum_{j_1,\dots ,j_d} T_{j_1\dots j_d}x_{j_1}\dots x_{j_d}$$
$$= \sum_{i_1 +\cdots +i_n = d}\binom{d}{i_1, \dots , i_n}T_{\scriptsize{\underbrace{1\dots1}_{i_1\text{ times}}\dots\underbrace{n\dots n}_{i_n\text{ times}}}}x_1^{i_1}\dots x_n^{i_n} = \sum_{i_1+ \cdots + i_n=d} \frac 1{i_1!\dots i_n!} u_{i_1,\dots,i_n} x_1^{i_1}\dots x_n^{i_n},$$
where
$$u_{i_1,\dots ,i_n} = d! T_{\scriptsize{\underbrace{1\dots 1}_{i_1\text{ times}}\dots\underbrace{n\dots n}_{i_n\text{ times}}}}.$$
We write the equations defining the odeco variety in terms of the variables $u_{i_1,\dots ,i_n}$. Note that for all such variables $i_1+\cdots + i_n = d$.

\begin{conjecture}\label{conj:equations} The prime ideal of the odeco variety inside $S^d \left(\mathbb C^n\right)$ is generated by
\begin{align}\label{odecoEquations1}
\sum_{s=1}^n u_{y+e_s}u_{v+e_s} - u_{w+e_s}u_{z+e_s} = 0,
\end{align}
where $y,v,w,z\in\mathbb Z_{\geq 0}^n$ are such that $\sum_iy_i = \sum_i v_i = \sum_i z_i = \sum_i w_i = d-1$ and $y+v = z+w$.
\end{conjecture}

Written in terms of the $T$-variables, these equations can be expressed as
\begin{align}\label{Tequations}
\sum_{s=1}^nT_{i_1,\dots, i_{d-1}, s}T_{j_1,\dots,j_{d-1},s} - T_{k_1,\dots,k_{d-1},s}T_{l_1,\dots,l_{d-1},s} = 0,
\end{align}
for all indices such that $\{i_r, j_r\} = \{k_r, l_r\}$, and also up to permuting the indices due to the fact that $T$ is symmetric.

Another way to think about \eqref{Tequations} is as follows. Suppose we contract $T$ along one of its dimensions, say the $d$-th dimension resulting into a tensor $T*_dT\in S^2(S^{d-1}(\mathbb R^n))$ whose entry indexed by $i_1,\dots, i_{d-1}, j_1,\dots,j_{d-1}$ is
$$(T*_dT)_{i_1,\dots, i_{d-1}, j_1,\dots,j_{d-1}} = \sum_{s=1}^nT_{i_1,\dots, i_{d-1}, s}T_{j_1,\dots,j_{d-1},s}.$$
Then, the equations \eqref{Tequations} are equivalent to saying that $T*_dT$ also lies inside $S^{2(d-1)}(\mathbb R^n)$.

\begin{example} When $d=2$ the elements of $S^2\left(\mathbb R^n\right)$ are symmetric matrices and the set of equations (\ref{odecoEquations1}) is empty, which is equivalent to the fact that all symmetric matrices are odeco.
\end{example}

In essence, the ideal defined by (\ref{odecoEquations1}) is a lifting of the toric ideal defining the Veronese variety $v_{d-1}\left(\mathbb C^n\right)\subset S^{d-1}\left(\mathbb C^n\right)$ to non-toric equations on $S^d\left(\mathbb C^n\right)$.

\begin{example}\label{example:33} Let $d = n = 3$. We will illustrate how to obtain the equations (\ref{odecoEquations1}) of the odeco variety in $S^3\left(\mathbb C^3\right)$ from the equations of the Veronese variety $v_{d-1}\left(\mathbb C^n\right) = v_2\left(\mathbb C^3\right)$. 
Consider the Veronese embedding $v_2: \mathbb C^3 \to S^2\left(\mathbb C^3\right)$ given by $x \mapsto x^{\otimes 2}$. The image $v_2\left(\mathbb C^3\right)$ is the set of rank one $3\times 3$ symmetric matrices. The space $S^2\left(\mathbb C^3\right)$ has coordinates $u_{i_1i_2i_3}$, where $i_1+i_2 + i_3 = 2$. There are six equations that define the prime ideal of the Veronese variety $v_2\left(\mathbb C^3\right)\subseteq S^2\left(\mathbb C^3\right)$ and they are
\begin{align}\label{VeroneseEquations}
&u_{200}u_{020} - u_{110}^2~=~0, \hspace{1cm} u_{200}u_{011} - u_{110}u_{101}~=~0, \notag \\
&u_{200}u_{002} - u_{101}^2~=~0,  \hspace{1cm} u_{110}u_{002} - u_{101}u_{011}~=~0,\\ 
&u_{101}u_{020} - u_{110}u_{011}~=~0,  \hspace{1cm}u_{020}u_{002} - u_{011}^2~=~0.\notag
\end{align}
Each of these equations has the form $u_yu_v - u_wu_z = 0$, where $y,v,w,z\in\mathbb Z^3_{\geq0}$, $\sum_iy = \sum_iv=\sum_iw=\sum_iz = 2$, and $y+v = w+z$. Each such equation leads to one of the equations in (\ref{odecoEquations1}) as follows
$$u_yu_v - u_wu_z\mapsto u_{y+e_1}u_{v+e_1} - u_{w+e_1}u_{z+e_1} + u_{y+e_2}u_{v+e_2} - u_{w+e_2}u_{z+e_2} + u_{y+e_3}u_{v+e_3} - u_{w+e_3}u_{z+e_3}.$$
Therefore, using (\ref{VeroneseEquations}), we obtain the six equations in (\ref{odecoEquations1})
\begin{align*}
u_{200}u_{020} - u_{110}^2 \hspace{0.2cm}&\mapsto \hspace{0.2cm}u_{300}u_{120} - u_{210}^2 + u_{210}u_{030} - u_{120}^2 + u_{201}u_{021} - u_{111}^2,\\
 u_{200}u_{011} - u_{110}u_{101}\hspace{0.2cm}&\mapsto\hspace{0.2cm} u_{300}u_{111} - u_{210}u_{201} + u_{210}u_{021} - u_{120}u_{111} + u_{201}u_{012} - u_{111}u_{102},\\
u_{200}u_{002} - u_{101}^2 \hspace{0.2cm}&\mapsto\hspace{0.2cm} u_{300}u_{102} - u_{201}^2 + u_{210}u_{012} - u_{111}^2 + u_{201}u_{003} - u_{102}^2,\\
 u_{110}u_{002} - u_{101}u_{011}\hspace{0.2cm}&\mapsto\hspace{0.2cm} u_{210}u_{102} - u_{201}u_{111} + u_{120}u_{012} - u_{111}u_{021} + u_{111}u_{003} - u_{102}u_{012},\\
u_{101}u_{020} - u_{110}u_{011}\hspace{0.2cm}&\mapsto\hspace{0.2cm} u_{201}u_{120} - u_{210}u_{111} + u_{111}u_{030} - u_{120}u_{021} + u_{102}u_{021} - u_{111}u_{012},\\
u_{020}u_{002} - u_{011}^2\hspace{0.2cm}&\mapsto\hspace{0.2cm} u_{120}u_{102} - u_{111}^2 + u_{030}u_{012} - u_{021}^2 + u_{021}u_{003} - u_{012}^2.
\end{align*}
\end{example}

\begin{lemma}{\label{lem:derivatives}} The equations (\ref{odecoEquations1}) vanish on the odeco variety.
\end{lemma}
%
%
%

\begin{proof}[Proof of Lemma \ref{lem:derivatives}]
Let $T = \sum_i \lambda_i v_i^{\otimes d}$ be odeco. Then, by definition of the $u$-variables, at the point $T$
$$u_{y_1\dots y_n} = d!\sum_{i=1}^n \lambda_i v_{i1}^{y_1}\cdots v_{i_n}^{y_n} = d! \sum_{i=1}^n\lambda_iv_i^y.$$
Thus, at the point $T$, the equations (\ref{odecoEquations1}), for $y,v,w,z\in\mathbb Z_{\geq 0}^n$ with $y+v = w+z$ and $\sum_i y = \sum_iv = \sum_i w = \sum_i z = d-1$, have the form
\begin{align*}
&\sum_({s=1}^n u_{y+e_s}u_{v+e_s} - u_{w+e_s}u_{z+e_s}=\\
&= (d!)^2\sum_{s=1}^n \big(\sum_{i=1}^n\lambda_i v_i^{y+e_s}\big) \big( \sum_{j=1}^n\lambda_j v_j^{v+e_s}\big) - \big (\sum_{i=1}^n\lambda_i v_i^{w + e_s}\big) \big( \sum_{j=1}^n \lambda_j v_j^{z + e_s}\big)\\
& = (d!)^2 \sum_{s=1}^n \big( \sum_{i=1}^n \lambda_i^2 (\cancel{v_i^{y+v+2e_s}} - \cancel{v_i^{w+z+2e_s}} )+ \sum_{i\neq j}\lambda_i \lambda_j (v_i^{y+e_s}v_j^{v+e_s} - v_i^{w+e_s}v_j^{z+e_s}))\big)\\
& = (d!)^2\sum_{i\neq j} \lambda_i\lambda_j ( v_i^yv_j^v - v_i^w v_j^z){\sum_{s=1}^nv_{is}v_{js}}= 0,
\end{align*}
where the last row is 0 since $v_i$ and $v_j$ are orthogonal and $\sum_{s=1}^nv_{is}v_{js} = v_i\cdot v_j = 0$

Therefore, (\ref{odecoEquations1}) vanish on the odeco variety.
\end{proof}

We are going to select a subset of the equations (\ref{odecoEquations1}) that spans the vector space defined by (\ref{odecoEquations1}). More precisely, consider
\begin{align}\label{odecoEquations2}
f_{y, v, i, j} = \sum_{s=1}^n u_{y+e_s}u_{v+e_s} - u_{y+e_i-e_j + e_s}u_{v-e_i+e_j + e_s},
\end{align}
for all $i\neq j\in\{1,2,\dots ,n\}$ and all $y, v\in \mathbb Z_{\geq 0}^n$ whose entries sum to $d-1$ and $y_j \geq1$, $v_{i}\geq 1$.

We now prove Conjecture \ref{conj:equations} for the case $n=2$.

\begin{theorem}\label{prop:n=2}
When $n=2$, the equations (\ref{odecoEquations2}) form a Gr\"obner basis with respect to the term order $\prec$ (defined below as a refinement of the weight order \eqref{weightedOrder}) and the dimension of the variety they cut out is $\binom {n+1}2 = 3$. The ideal defined by (\ref{odecoEquations2}) is the prime ideal of the Odeco variety. \end{theorem}
\begin{proof}
We are going to work over the polynomial ring
$$\mathbb C[\mathbf u] := \mathbb C[u_{i_1i_2} | i_1,i_2\geq 0\text{ and }i_1 +i_2 = d]$$
$$ = \mathbb C[u_{d0}, u_{\left(d-1\right) 1}, \dots , u_{0 d}].$$
Then, the equations (\ref{odecoEquations2}) are
$$f_{y,v, 1, 2} =  u_{y + e_1}u_{v+e_1} - u_{y+e_1 -e_2 + e_1}u_{v - e_1 + e_2+e_1} +  u_{y + e_2}u_{v+e_2} - u_{y+e_1 -e_2 + e_2}u_{v - e_1 + e_2+e_2},$$
where $y, v\in\mathbb Z_{\geq 0}^2$, the sum of the entries of each of $y$ and $v$ is $d-1$ and $y_2\geq 1, v_1\geq1$. Let the ideal they generate be
\begin{align}\label{idealn2}I := \langle f_{y,v,1,2} | y,v\in\mathbb Z_{\geq 0}^2, \sum_{i} y_i = \sum_{i} v_i = d-1, y_2 \geq 1, v_1\geq 1\rangle.\end{align}

We introduce the following weights on our variables. Let
\begin{align}\label{weightedOrder}\text{weight}\left(u_{i \left(d-i\right)}\right) = i,\end{align}
for all $i=0,1,\dots ,d$. 
Consider the weighted term order on monomials $\prec$ given by the above weights, refined by the lexicographic term order such that $u_{d0} \succ u_{\left(d-1\right)1}\succ \cdots \succ u_{0d}$ in case of equal weights.

\bigskip

\textit{We first show that the equations (\ref{odecoEquations2}) form a Gr\"obner basis with respect to $\prec$.} Using Macaulay2, we have shown that they form a Gr\"obner basis for $d=1,2,\dots ,9$. Now, consider any $d > 9$. Take $f_{y',v',1,2}$ and $f_{y'',v'',1,2}$. By Buchberger's second criterion, we only need to consider the two polynomials when their initial terms have a common variable. Then, the two polynomials $f_{y',v',1,2}$ and $f_{y'',v'',1,2}$ contain $l\leq 9$ different variables in total. If we restrict our generators (\ref{odecoEquations2}) to these $l$ variables only, the restriction of the term order is the same as the term order in the case $d=l-1$, and we have shown that in this case, the restricted generators form a Gr\"obner basis. Therefore, we can reduce the S-pair of $f_{y',v',1,2}$ and $f_{y'',v'',1,2}$ to $0$ using the generators (\ref{odecoEquations2}). Thus, the equations (\ref{odecoEquations2}) form a Gr\"obner basis.

\bigskip
\textit{Next, we show that the ideal $I$ generated by (\ref{odecoEquations2}) has dimension 3.} One way to see this is to use Lemma \ref{lem:dimension} together with the fact that $I$ is prime, which is proven below. Another way to see that $\dim I =3$ is to reason with standard monomials as follows.

Note that because of our choice of term order $\prec$, the initial term of every $f_{u,v,1,2}$ is square-free. The reason is that if $u_{y+e_s} = u_{v+e_s}$, then, weight$(u_{y+e_1}u_{v+e_1})$ $ = $  weight$($ $ u_{y+e_1-e_2+e_1}$ $u_{v-e_1+e_2-e_1}) > $ weight$($ $u_{y+e_2}u_{v+e_2}) $ $= $ weight$(u_{y+e_1-e_2+e_2}$ $u_{v-e_1+e_2-e_2})$, but $u_{y+e_1-e_2+e_1}$ appears first in $\prec$, so, $u_{y+e_1-e_2+e_1}u_{v-e_1+e_2-e_1}$ is the leading term. The reasoning is similar if $u_{y+e_1-e_2+e_2} = u_{v-e_1+e_2-e_1}$. Therefore, in$_{\prec}I$ (and thus $I$) is a radical ideal.

To show that $\dim I = 3$, let $S = \{u_{i_1 \left(d-i_1\right)}, u_{i_2 \left(d-i_2\right)}, u_{i_3 \left(d-i_3\right)}, u_{i_4 \left(d-i_4\right)}\}$ be a set of four variables, where $i_1 > i_2>i_3>i_4$. We will show that there is a monomial with only variables from $S$ which is not standard. This would mean that $\dim I\leq 3$. Indeed, consider
$$f_{\left(i_1-1,d-i_1+1\right),\left(i_3+1, d-i_3-1\right),1,2} = u_{\left(i_1-1\right)\left(d-i_1+1\right)}u_{\left(i_3+1\right)\left(d-i_3+1\right)} - \underline{u_{i_1\left(d-i_1\right)}u_{i_3\left(d-i_3\right)}} $$
$$+ u_{\left(i_1-2\right)\left(d-i_1+2\right)}u_{i_2\left(d-i_2\right)} - u_{\left(i_1-1\right)\left(d-i_1+1\right)}u_{\left(i_2-1\right)\left(d-i_2+1\right)}.$$
Since $i_1-2\geq i_3$, the initial term is $u_{i_1\left(d-i_1\right)}u_{i_3\left(d-i_3\right)}$. Therefore, $\dim I\leq 3$.

Now, consider the set $S = \{u_{2\left(d-2\right)}, u_{1\left(d-1\right)}, u_{0d}\}$. Suppose there exists
$$f_{y,v, 1, 2} =  u_{y + e_1}u_{v+e_1} - u_{y+e_1 -e_2 + e_1}u_{v - e_1 + e_2+e_1} +  u_{y + e_2}u_{v+e_2} - u_{y+e_1 -e_2 + e_2}u_{v - e_1 + e_2+e_2},$$
such that in$_\prec \left(f\right)$ has both of its variables in $S$. We know that in$_{\prec}\left(f\right) = u_{y + e_1}u_{v+e_1}$ or in$_{\prec}\left(f\right) = u_{y+e_1 -e_2 + e_1}u_{v - e_1 + e_2+e_1}$. Moreover, if $y = \left(y_1,y_2\right)$ and $v = \left(v_1,v_2\right)$, then, $y_2,v_1 \geq 1$ and $y_1,v_2\leq d-2$. Thus, if in$_{\prec}\left(f\right) = u_{y + e_1}u_{v+e_1}$ and $u_{y + e_1},  u_{v+e_1}\in S$, then, $ v = \left(1,d-2\right)$ and $y = \left(1,d-2\right)$ or $y=\left(0,d-1\right)$. Since $f_{y,v,1,2}$ is not the trivial polynomial $0$, then, $y\neq\left(0,d-1\right)$. Thus, $y=\left(1,d-2\right)$. But this is impossible since in$_\prec \left(f\right)$ is square-free for every generator $f$. If in$_\prec \left(f\right) =  u_{y+e_1 -e_2 + e_1}u_{v - e_1 + e_2+e_1}$ and $u_{y+e_1 -e_2 + e_1},u_{v - e_1 + e_2+e_1}\in S$, then, $u_{\left(y_1+2,y_1-1\right)}\in S$. But $y_1\geq 1$, so, $y_1+2\geq 3$, therefore, $u_{\left(y_1+2,y_2-1\right)}\not\in S$. In any case, there can't be a monomial with only variables in $S$, which is a leading term of an element in $I$. Thus, $\dim I = 3$.

Another way to see that $\dim I\geq 3$ is by noting that $V\left(I\right)$ contains the odeco variety, which has dimension $3$ in this case.

\bigskip
\textit{Finally, we show that the ideal generated by (\ref{odecoEquations2}) is prime.} Let $J$ be the ideal generated by the leading binomials of the elements in (\ref{odecoEquations2}) with respect to the weight order defined by \eqref{weightedOrder} (without considering the refinement given by the order of the variables). Denote by $g_w$ the leading term of a polynomial $g$ just with respect to this weight order. Then, $\left(f_{y,v,1,2}\right)_w = u_{y + e_1}u_{v+e_1} - u_{y+e_1 -e_2 + e_1}u_{v - e_1 + e_2+e_1}$, and $J=\langle u_{y + e_1}u_{v+e_1} - u_{y+e_1 -e_2 + e_1}u_{v - e_1 + e_2+e_1}: y,v\in\mathbb Z_{\geq 0}^2, y_1+y_2 = v_1+v_2=d-1, y_2,v_1\geq1\rangle$. The ideal $J$ is the prime ideal of the rational normal curve; in particular, it is prime. Moreover, by Proposition 1.13 in \cite{S}, in$_\prec\left( I\right) = $in$_\prec\left(J\right)$.  Therefore, in$_\prec\left(I\right)$ is an initial ideal of both $I$ and $J$. In the following paragraph, we show that $J$ is the initial ideal of $I$ with respect to the weight order given by \eqref{weightedOrder}. Then, since $J$ is prime, it follows that $I$ is prime.

Suppose $J$ is not initial, i.e. there exists $g\in I$ such that $g_w\not\in J$. Choose $g$ with in$_\prec\left(g\right)$ as small as possible. Since the elements $f_{u,v,1,2}$ form a Gr\"obner basis of $I$, then, there exist $y,v$ such that in$_\prec\left(g\right)$ is divisible by in$_\prec \left(f_{y,v,1,2}\right)$. Then, $g=\alpha_{y,v} f_{y,v,1,2} + g_1$, where $\alpha_{y,v}$ is a monomial and in$_\prec g_1\prec$ in$_\prec g$. But note that then, $g_w = \alpha_{y,v}\left(u_{y + e_1}u_{v+e_1} - u_{y+e_1 -e_2 + e_1}u_{v - e_1 + e_2+e_1}\right) + \left(g_1\right)_w$. Since $u_{y + e_1}u_{v+e_1} - u_{y+e_1 -e_2 + e_1}u_{v - e_1 + e_2+e_1}\in J$ and $g_w\not\in J$, then, $\left(g_1\right)_w\not\in J$. But this is a contradiction since in$_\prec\left(g_1\right)\prec$ in$_\prec\left(g\right)$ and we chose in$_\prec\left(g\right)$ to be as small as possible such that $g_w\not \in J$.

Therefore, $J$ is initial. Since it is prime, then, $I$ is also prime. By Lemma \ref{lem:dimension}, the dimension of the odeco variety for $n=2$ is $3$. Moreover, it is contained in $\mathcal V\left(I\right)$. Since $\mathcal V\left(I\right)$ is also irreducible and has dimension $3$, then, $I$ is exactly the prime ideal of the Odeco variety.
\end{proof}

\subsection{Evidence for Conjecture \ref{conj:equations}}\label{sec:evidence}

\begin{lemma}\label{lem:dimension} The odeco variety is an irreducible component of $\mathcal V\left(I\right)$, where $I$ is the ideal generated by the equations (\ref{odecoEquations1}). \end{lemma}

\begin{proof}
We show that the dimension of the component of $\mathcal V\left(I\right)$ containing the odeco variety is equal to $\binom {n+1}2$. This equals the dimension of the odeco variety. Since it is irreducible, then it is an irreducible component of $\mathcal V\left(I\right)$.

Consider the point $T\in\mathcal V\left(I\right)$ given by $T_{i\dots i} = 1$ for all $i=1,\dots ,n$ and all other entries of $T$ are $0$. The polynomial corresponding to $T$ is the standard Fermat polynomial $f_T\left(x_1,\dots ,x_n\right) = x_1^d + \cdots + x_n^d$. In the $u$ coordinates, $T$ is represented by the point for which $u_{0\dots 0 d 0\dots 0} = u_{d e_i} = 1$ for $i=1,\dots ,n$ and all other $u_{i_1\dots i_n} = 0$. 

We can select generators $f_{v, w}$ for $I$ such that $v,w\in\mathbb Z_{\geq 0}^n$ with $\sum_iv_i = \sum_i w_i = d-1$ and

$$f_{v, w} = \sum_{i=1}^s u_{v+e_s}u_{w+e_s} - u_{\text{sort}\left(v,w\right)_1 + e_s}u_{\text{sort}\left(v, w\right)_2 + e_s},$$
where sort$\left(v,w\right)_1$ and sort$\left(v,w\right)_2$ are defined as follows. Given $v$ and $w$, form the corresponding sequences $t\left(v\right) = \underbrace{1\dots 1}_{v_1\text{ times}}\underbrace{2\dots 2}_{v_2\text{ times}}\dots \underbrace{n\dots n}_{v_n\text{ times}}$ and $t\left(w\right)=\underbrace{1\dots 1}_{w_1\text{ times}}\underbrace{2\dots 2}_{w_2\text{ times}}\dots \underbrace{n\dots n}_{w_n\text{ times}}$. Let $t\left(v,w\right) = \text{sort}\left(t\left(v\right)\cup t\left(w\right)\right)$ be the sequence obtained by concatenating $t\left(v\right)$ and $t\left(w\right)$ and then sorting. Let $t\left(v,w\right)_{1}$ be the subsequence of elements in odd positions and $t\left(v,w\right)_{2}$ the subsequence of elements in even positions. Define $u_{\text{sort(v,w)}_1}$ and $u_{\text{sort}\left(v,w\right)_2}$ be the corresponding $u$ variables. The fact that the polynomials $f_{u,w}$ generate $I$ follows from Theorem 14.2 in \cite{S}.

We form the Jacobian $\mathcal J$ of $I$ at the point $T$. Index the rows of $\mathcal J$ by the generators $f_{v,w}$ and index the columns by the variables $u_{i_1,\dots ,i_n}$. Note that $\frac{\partial f}{\partial u_{d e_i}}|_T = 0$ since the monomials in $f_{v,w}$ containing $u_{d e_i}$ contain another variable $u_{i_1,\dots ,i_n}\neq u_{d e_j}$ for all $j=1,\dots ,n$. Therefore, the column corresponding to $u_{d e_i}$ is zero.

Note that the monomials $u_{\text{sort}\left(v,w\right)_1 + e_s}u_{\text{sort}\left(v, w\right)_2 + e_s}$ cannot contain a variable $u_{d e_i}$ for any $v$ and $w$ that give a nontrivial $f_{u,v}$, so they  don't matter in the Jacobian analysis.

Now, the column of $\mathcal J$ corresponding to the variable $u_{\left(d-1\right)e_i + e_j}$ for $i\neq j$ has $1$ only in the rows corresponding to $f_{\left(d-1\right)e_i, \left(d-1\right)e_j}$ and so does the variable $u_{\left(d-1\right)e_j + e_i}$. Therefore, the variables $u_{\left(d-1\right)e_i + e_j}$ and the polynomials $f_{\left(d-1\right)e_i, \left(d-1\right)e_j}$ form a block in $\mathcal J$ of rank $\binom n2$, which equals the number of pairs $i\neq j$.

For any other variable $u_{i_1,\dots ,i_n}$, such that $\left(i_1,\dots ,i_n\right)\neq d e_i$ or $\left(d-1\right)e_i + e_j$, its corresponding column is nonzero only at the rows corresponding to the polynomials $f_{\left(i_1,\dots ,i_n\right)-e_s, \left(d-1\right)e_s}$ for all $s$ such that $i_s > 0$. Each such polynomial has no other $1$'s in its row except for the one at $u_{i_1,\dots ,i_n}$. Therefore, each variable $u_{i_1,\dots ,i_n}$, such that $\left(i_1,\dots ,i_n\right)\neq d e_i$ or $\left(d-1\right)e_i + e_j$, contributes a size $1\times \{\# s: i_s > 0\}$ nonzero block to $\mathcal J$, so it contributes 1 to the rank. Therefore, the rank of $\mathcal J$ is
$$\# \text{ variables } - \#\{u_{de_i}\} - \#\{u_{\left(d-1\right)e_i + e_j : i\neq j}\} + \binom n2$$
$$ = \#\text{ variables } - n - n\left(n-1\right) + \binom n2 = \#\text{ variables } - \binom{n+1}2.$$
Thus, the rank of the Jacobian at a smooth point in the irreducible component of $T$ is at least $ \#\text{ variables } - \binom{n+1}2$, so  the dimension of an irreducible component containing $T$ is at most $\binom{n+1}2$.

Since the odeco variety is irreducible, has dimension $\binom{n+1}2$, contains $T$, and is contained in $\mathcal V\left(I\right)$, then it is one of the irreducible components of $\mathcal V\left(I\right)$.
\end{proof}
 Lemma \ref{lem:dimension} shows that one only needs to show that the ideal $I$ is prime in order to confirm Conjecture \ref{conj:equations}.

\subsection*{Computations}
In Figure~\ref{Fig2} we show some computational checks of the conjecture. 

Since the ideal $I$ becomes quite large, as $n$ and $d$ grow, it soon becomes hard to check its primality. It was easy to check the conjecture was correct in the case $n=d=3$ using Macaulay2. The case $n=3,d=4$ was checked using the numerical homotopy software Bertini. We were unable to confirm the rest of the results using (short) computations.

\begin{figure}[H]
\begin{center}
\begin{tabular} {|c|c|c|c|c|c|}
\hline
$n$ & $d$ & dimension & degree & \# min. gens. & conjecture check \\ \hline
3 & 3 & 6 & 10& 6& True\\ \hline
3 & 4 & 6&35&27& True\\ \hline
3 & 5 & 6&84&75& \\ \hline
4 & 3 & $\geq 10$ &&20& \\ \hline
4 & 4 & $\geq 10$ &&126& \\ \hline
5 & 3 & $\geq 15$ &&50&\\ \hline
\end{tabular}
\caption{A table of what can be found computationally about the ideal $I$ generated by the equations in \eqref{odecoEquations1}.}\label{Fig2}
\end{center}
\end{figure}

In upcoming work with Jan Draisma, Emil Horobet and Ada Boralevi, we show that a real symmetric tensor satisfies the proposed equations if and only if it is odeco. A complete proof of Conjecture~\ref{conj:equations} is still in progress.

\section*{Acknowledgements}
I would like to thank my advisor Bernd Sturmfels for his great help in this project. I would also like to thank Kaie Kubjas and Luke Oeding for helpful comments and Matthew Niemerg for his help with the software Bertini. The author was supported by a UC Berkeley Graduate Fellowship and by the National Institute of Mathematical Sciences (NIMS) in Daejeon, Korea.


\newcommand{\Addresses}{{
  \bigskip
  \footnotesize

 Author's address: \textsc{755 Evans Hall, Department of Mathematics,
    University of California, Berkeley, Berkeley, CA94720}\par\nopagebreak
  \textit{E-mail} \texttt{erobeva@berkeley.edu}
}}

\Addresses


\begin{thebibliography}{15}
\bibitem{AGHKT}A.~Anandkumar, R.~Ge, D.~Hsu, S.~Kakade, and M.~Telegarsky.
Tensor Decompositions for Learning Latent Variable Models. 
{\em Journal of Machine Learning Research (2012).}

\bibitem{AHK}A.~Anandkumar, D.~Hsu, and S.~Kakade.
A Method of Moments for Mixture Models and Hidden Markov Models.
{\em Twenty-Fifth Annual Conference on Learning Theory (2012).}

\bibitem{BCMT}J.~Brachat, P.~Common, B.~Mourrain, and E.~Tsigaridas.
Symmetric Tensor Decomposition.
{\em Linear Algebra and Applications 433:11-12 (2010) 851-872}

\bibitem{CS}D.~Cartwright and B.~Sturmfels.
The Number of Eigenvalues of a Tensor.
{\em Linear Algebra and its Applications, 432:2 (2013) 942-952 .}

\bibitem{CGLM} P. Comon, G. Golub, L.-H. Lim, and B. Mourrain, 
Symmetric tensors and symmetric tensor rank.
{\em SIAM J. Matrix Anal. Appl., 30:3 (2008) 1254-1279 }

\bibitem{ES}D.~Eisenbud and B.~Sturmfels.
Binomial Ideals.
{\em Duke Mathematical Journal 84 (1996) 1-45}

\bibitem{HL}C.~Hillar and L.-H.~Lim.
Most Tensor Problems are NP Hard.
{\em Journal of the ACM 60:6 (2013) Art. 45 }

\bibitem{K03} T.~Kolda.
A Counterexample to the Possibility of an Extension of the Eckart-Young Low-Rank Approximation Theorem for the Orthogonal Rank Tensor Decomposition.
{\em SIAM J. Matrix Anal. Appl. 24:3 (2003) 762-767}

\bibitem{K01} T.~Kolda.
Orthogonal Tensor Decompositions.
{\em SIAM J. Matrix Anal. Appl. 23:1 (2001) 243-255}

\bibitem{L} J.~M.~Landsberg.
Tensors: Geometry and Applications.
{\em Graduate Studies in Mathematics, American Mathematical Society (2011)}

\bibitem{LanO}J.~M.~Landsberg and G.~Ottaviani.
Equations for Secant Varieties of Veronese and Other Varieties.
{\em Annali di Matematica Pura ed Applicata 192:4 (2013) 569-606}

\bibitem{Lim}L.-H.~Lim.
Singular Values and Eigenvalues of Tensors: a Variational Approach.
{\em Computational Advances in Multi-Sensor Adaptive Processing, 2005 1st IEEE International Workshop (2005) 129-132}

\bibitem{LO} L.~Oeding and G.~Ottaviani.
Eigenvectors of Tensors and Algorithms for Waring Decomposition.
{\em Journal of Symbolic Computation 54 (2013) 9-35}

\bibitem{R} C.~Raicu.
Secant Varieties of Segre-Veronese Varieties.
{\em Algebra and Number Theory 6:8 (2012) 1817-1868}

\bibitem{S} B.~Sturmfels.
Gr\"obner Bases and Convex Polytopes.
{\em University Lecture Series, American Mathematical Society (1996)}

\bibitem{Qi}L.~Qi.
Eigenvalues of a Real Symmetric Tensor.
{\em Journal of Symbolic Computation 40:6 (2005) 1302-1324}

\end{thebibliography}
\end{document}